\documentclass[12pt,twoside]{amsart}
\usepackage{latexsym,amsmath,amsopn,amssymb,amsthm,amsfonts}
\usepackage[T2A]{fontenc}
\usepackage[cp1251]{inputenc}
\usepackage[english]{babel}
\usepackage{graphicx}

\newtheorem{theorem}{Theorem}
\newtheorem{corollary}{Corollary}
\newtheorem{proposition}{Proposition}

\newtheorem{definition}{Definition}

\newcommand{\ch}{\operatorname{ch}}
\newcommand{\sh}{\operatorname{sh}}
\newcommand{\tth}{\operatorname{th}}

\begin{document}

\title[Event horizon]{Correct observer's event horizon in de Sitter space-time}
\author{{V.N.Berestovskii, I.A.Zubareva}}%
\thanks{The first author is partially supported by Grants of Russian Federation Government for state support of scientific investigations (Agreement no. 14.B25.31.0029) and of RFBR 11-01-00081-a}

\vspace{1cm} \maketitle {MSC 2010: 51P05, 53C50, 53C99, 83C15, 83F05}
\vspace{1cm} \maketitle {\small
\begin{quote}
\noindent{\sc Abstract.}
A correct description of the past and the
future event horizons for timelike geodesics in de Sitter
space-time of the first kind is given.
\end{quote} }

\maketitle {\small
\begin{quote}
\end{quote} }

\section{Introduction and main results}
\label{in}

In this paper we continue to study properties of the future and the past event horizons for any time-like geodesic
in de Sitter space-time of the first kind $S(R)$ announced earlier in the paper~\cite{Ber}. Later we shall use shorter
term ''de Sitter space-time''.

In Section~\ref{pr} we give necessary definitions and known results about globally
hyperbolic space-time.

In Section~\ref{osn} are given exact description of the past event horizon for every time-like geodesic of de Sitter
space-time. We prove the theorem announced earlier in the paper ~\cite{Ber} as the theorem~8. This theorem sets
a connection of the future and the past event horizons between themselves and with the so-called Lobachevsky
space of positive curvature $\frac{1}{R^2}$ in the sense of B.A.\,Rosenfeld (see p. 155 in~\cite{Ros}),
which is obtained by the gluing of antipodal events in $S(R)$. In this paper we give correct figures (see Fig.\,2, \,3)
of an observer's event horizon, which differ with the corresponding non-correct figure (see. Fig.\,1) from the
Hawking's book~\cite{Hokrus} on the page~120.

\section{Preliminaries}
\label{pr}

We now remind necessary definitions from \cite{Beem}, \cite{Hok}.

Let $M$ be a $C^{\infty}$--manifold of dimension $n+1\geq 2$. A \textit{Lorentzian metric} $g$ for $M$ is a
smooth symmetric tensor field of the type $(0,2)$ on $M$ which assigns to each point $p\in M$ a nondegenerate inner
product $g\mid_{p}:T_pM\times T_pM\rightarrow\mathbb{R}$ of the signature $(+,+,\dots, +,-).$
Then the pair $(M,g)$ is said to be \textit{Lorentzian manifold}.

A nonzero tangent vector $v$ is said to be respectively \textit{time-like, space-like}, or \textit{isotropic}
if $g(v,v)$ is negative, positive, or zero. A tangent vector $v$ is said to be  \textit{non-space-like} if it is
time-like or isotropic. A continuous vector field $X$ on Lorentzian manifold $M$ is called \textit{time-like} if
$g(X(p),X(p))<0$ for all events $p\in M.$ If Lorentzian manifold  $(M,g)$ admits
a time-like vector field $X$, then we say that $(M,g)$ is \textit{time oriented by the field $X$}. The time-like vector
field $X$ separates all non-space-like vectors into two disjoint classes of \textit{future directed} and
\textit{past directed} vectors. More exactly,  a non-space-like vector $v\in T_pM,\,\,p\in M,$ is said to be
\textit{future directed} (respectively, \textit{past directed}) if $g(X(p),v)<0$ (respectively, $g(X(p),v)>0$).
A Lorentzian manifold is \textit{time orientable} if it admits some time-like vector field $X$.

\begin{definition}
\label{spacetime} A space-time $(M,g)$ is a connected  Hausdorff  manifold of dimension equal or greater than two
with a countable basis supplied with a Lorentzian  $C^{\infty}$-metric $g$ and some time orientation.
\end{definition}

A continuous piecewise smooth curve (path) $c=c(t)$ with $t\in [a,b]$ or $t\in (a,b)$ on Lorentzian manifold
$(M,g)$ is said to be \textit{non-space-like} if $g(c'_{l}(t),c'_{r}(t))\leq 0$ for every inner point $t$ from
the domain of the curve $c$, where $c'_{l}(t)$ (respectively, $c'_{r}(t)$) denotes left (respectively, right) tangent
vector. If $(M,g)$ is a space-time, then the curve $c=c(t)$ with  $t\in [a,b]$ or $t\in (a,b),$ is
\textit{future directed} or \textit{past directed}, i.e. all (one-sided) tangent vectors of the curve $c$
either \textit{are directed to the future} or  \textit{are directed to the past}. The \textit{causal future}  $J^+(L)$
(respectively, the \textit{causal past} $J^-(L)$) of a subset
$L$ of the space-time $(M,g)$ is defined as the  set of all events $q\in M,$ for which  there exists future directed
(respectively past directed) curve $c=c(t), t\in [a,b],$ such that
 $c(a)\in L, c(b)=q.$ If $p\in M$, then we will use reduced notation $J^+(p)$ and $J^-(p)$ instead of
 $J^+(\{p\})$ and $J^-(\{p\})$.

\begin{definition}
\label{stronglycausal}\cite{Beem} An open set $U$ in a space-time
is said to be causally convex if no non-space-like curve intersects
$U$ in a disconnected set. The space-time $(M,g)$ is said to be strongly causal
if each event in $M$ has arbitrarily small causally convex neighborhoods.
\end{definition}

In \cite{Beem} it has been proved the following important proposition 2.7.

\begin{proposition}
\label{alsc}
A space-time $(M,g)$ is strongly causal if and only if the sets of the form
$I^+(p)\cap I^-(q)$ with arbitrary $p,q\in M$ form a basis of original topology
(i.e. the Alexandrov topology induced on $(M,g)$ agrees with given manifold topology).
\end{proposition}

\begin{definition}
\label{globgiperbolic} A space-time $(M,g)$ is called globally
hyperbolic if it is strongly causal and satisfies the condition that $J^+(p)\cap J^-(q)$ is compact for all $p,q\in M$.
\end{definition}

\begin{definition}
\label{horizons} Let $S$ be a subset in a globally hyperbolic space-time $(M,g).$ Then $\Gamma^{-}(S)$
(respectively, $\Gamma^{+}(S)$) denotes the boundary of the set $J^{-}(S)$ (respectively, $J^{+}(S)$) and is called
the past event horizon (respectively, the future event horizon) of the set $S$.
\end{definition}

A simplest example of a globally hyperbolic space-time is \textit{the Minkowski space-time}
\textit{Mink}$^{\,n+1}$, $n+1\geq 2,$ i.e. a
manifold $\mathbb{R}^{n+1}$, $n+1\geq 2,$ with the Lorentz metric  $g$ which has the constant components $g_{ij}$:
$$g_{ij}=0,\,\,\mbox{if}\,\,i\neq j;\quad g_{11}=\ldots
=g_{nn}=1,\,\,g_{(n+1)(n+1)}=-1.$$ in natural coordinates $(x_1,\dots,$ $x_n,$$t)$
on $\mathbb{R}^{n+1}.$ The time orientation is defined by the vector field
$X$ with components $(0,\dots,0,1)$ relative to canonical coordinates in $\mathbb{R}^{n+1}.$

A more interesting example is \textit{de Sitter space-time}. It is easily visualized as the hyperboloid
$S(R)$ with one sheet
\begin{eqnarray}\label{m1}
\sum_{k=1}^{n}x_k^2-t^2=R^2,\,\,R>0,
\end{eqnarray}
 in Minkowski space-time
\textit{Mink}$^{\,n+1}$, $n+1\geq 3,$ with Lorentzian metric induced from
\textit{Mink}$^{\,n+1}$.

\textit{The Lorentz group} is the group of all linear isometries of the space \textit{Mink}$^{\,n+1},$ transforming
to itself the <<upper>> sheet of the hyperboloid
$\sum_{k=1}^{n}x_k^2-t^2=-1$ (which is isometric to the Lobachevsky space of the constant sectional curvature -1).
The Lorentz group acts transitively by isometries on $S(R).$

The time orientation on $S(R)$ is defined by unit tangent vector field
$Y$ such that $Y$ is orthogonal to all time-like sections
$$
S(R,c)=S(R)\cap \{(x_1,\dots,x_{n}, t)\in \mathbb{R}^{n+1}:
t=c\},\,\,c\in\mathbb{R}.
$$

Notice that every integral curve of the vector field $Y$ is a future directed time-like geodesic in
$S(R)$. Therefore we can consider it as a world line of some observer.

\section{Main result}
\label{osn}

The main result of this paper is the following

\begin{theorem}
\label{gorizont}
 Let $L$ be a time-like geodesic in de Sitter space-time, $\Gamma^{-}(L)$ is the past event horizon
 for $L$ (observer's event horizon). Then

1. $\Gamma^{-}(L)=S(R)\cap\alpha$, where $\alpha$ is some hyperspace in $\mathbb{R}^{n+1}$ which goes through
the origin of coordinate system, consists of isotropic geodesics.

2. $J^+(L)=-J^-(-L),$ $J^-(L)=-J^+(-L).$

3. The sets $J^-(L)$ and $J^+(-L)$ (respectively $J^+(L)$ and
$J^-(-L)$) don't intersect and have joint boundary $\Gamma^{-}(L)$.
In particular, the past event horizon for $L$ coincides with the future event horizon for $-L$
(respectively, the future event horizon for $L$ coincides with the past event horizon for $-L$).

4. The quotient map $pr: S(R)\rightarrow S^1_n(R),$ gluing antipodal events in $S(R),$
is diffeomorphism on all open submanifolds $J^+(L),$ $J^-(-L),$ $J^-(L),$ $J^+(-L)$ which identifies antipodal
events of boundaries for these submanifolds (i.e. the corresponding event horizons).

5.  The quotient manifold $(S^1_n(R),G),$ where $g=pr^{\ast}G,$ is the Lobachevsky space of positive curvature
$\frac{1}{R^2}$ in the sense of B.A.Rosenfeld (see p. 155 in \cite{Ros}).
\end{theorem}

\begin{proof}
Notice that if $L$ and $L^{\prime}$ are time-like geodesics in $S(R)$ and $p\in L$,
$p^{\prime}\in L^{\prime},$ then there exists preserving the time direction isometry $i$ of de Sitter space-time
such that $i(p)=p^{\prime}$ and $i(L)=L^{\prime}$. Therefore $i$ translates the past event horizon of $L$ to
the past event horizon of $L^{\prime}$. Therefore it is enough to prove points~1--4 of
theorem\,\ref{gorizont} only for one time-like geodesic.

We will suppose that $L$ is integral curve of the vector field $Y$, which intersects $S(R,0)$ at the event $p$
with Descartes coordinates $(R,0,\dots,0)$. Let us show that
\begin{eqnarray}\label{m2}
\Gamma^{-}(L)=S(R)\cap\{(x_1,\dots,x_{n},t)\in \mathbb{R}^{n+1}:
x_1=t\}.
\end{eqnarray}

Denote by $C_p$ isotropic cone at the point $p$\,\,with Descartes coordinates
$(R,0,\dots,0)$. It follows from (\ref{m1}) that
\begin{eqnarray}\label{m3}
C_p=\{(x_1,\dots,x_n,t)\in S(R)\,\mid\,\,x_1=R\}.
\end{eqnarray}

Note that the causal past $J^{-}(p)$ of the event $p$ is the region in $S(R)$, lying in the hyperspace $t<0$ and bounding
by isotropic cone $C_p$.

For any $\psi\in\mathbb{R},$ the restriction of the Lorentz transformation $\Phi_\psi:S(R)\rightarrow S(R)$, realizing
hyperbolic rotation in two-dimensional plane $Ox_1t$, prescribed by formulae
\begin{eqnarray}\label{m4}
x^{\prime}_1=x_1\ch\psi+t\sh\psi;\quad
x^{\prime}_i=x_i,\,\,i=2,\dots,n; \quad
t^{\prime}=x_1\sh\psi+t\ch\psi,
\end{eqnarray}
is an isometry of the space-time $S(R)$, preserving the time direction. The set $\Phi$ of all such transformations
$\Phi_\psi$, $\psi\in\mathbb{R}$, forms a one-parameter subgroup of the Lorentz group. Here the orbit of the event
$p$ with respect to $\Phi$ coincides (up to parametrization) with the curve $L$, and the event
$L(\psi):=\Phi_{\psi}(p)$ has Descartes coordinates $(R\ch\psi,0,\ldots,0,R\sh\psi)$.

Under the action of $\Phi_\psi,$ isotropic cone $C_p$ at the point $p$ passes to isotropic cone $C_{L(\psi)}$ at the
point $L(\psi)$ and the causal past $J^{-}(p)$ of the event $p$ passes to the causal past $J^{-}(L(\psi))$
of the event $L(\psi)$. It follows from (\ref{m3}), (\ref{m4}) that
\begin{eqnarray}\label{m5}
C_{L(\psi)}=\left\{(x_1,\dots,x_n,t)\in
S(R)\,\mid\,\,x_1-t\tth\psi=\frac{R}{\ch\psi}\right\}.
\end{eqnarray}
Note that the set $J^{-}(L)$ of observed events for $L$ is the union of all sets $J^{-}(L(\psi))$, where
$\psi\in\mathbb{R}$. If $\psi_1<\psi_2$ then $L(\psi_1)\in J^{-}(L(\psi_2))$, therefore
$J^{-}(L(\psi_1))\subset J^{-}(L(\psi_2))$. Hence the past event horizon of $L$ is limiting position of the <<lower>>
half (lying in the hyperspace $t<\frac{x_1\sh{\psi}}{\ch{\psi}}$) of isotropic cone
$C_{L(\psi)}$ when $\psi\rightarrow +\infty$. Now (\ref{m2}) follows from (\ref{m5}) and the fact that
$\ch\psi\rightarrow +\infty$, $\tth\psi\rightarrow 1$ when
$\psi\rightarrow +\infty$.

Note also that the future event horizon $\Gamma^{+}(L)$ for $L$ is limiting position of  <<upper>> half
(lying in the hyperspace $t>\frac{x_1\sh{\psi}}{\ch{\psi}}$) of isotropic cone
$C_{L(\psi)}$ when $\psi\rightarrow -\infty$. Then by (\ref{m5}),
\begin{eqnarray}\label{m6}
\Gamma^{+}(L)=S(R)\cap\{(x_1,\dots,x_{n},t)\in \mathbb{R}^{n+1}:
x_1+t=0\}.
\end{eqnarray}

It follows from (\ref{m2}), (\ref{m6}) that
\begin{eqnarray}\label{m7}
J^{-}(L)=S(R)\cap\{(x_1,\dots,x_{n},t)\in \mathbb{R}^{n+1}:
x_1-t>0\},
\end{eqnarray}
\begin{eqnarray}\label{m8}
J^{+}(L)=S(R)\cap\{(x_1,\dots,x_{n},t)\in \mathbb{R}^{n+1}:
x_1+t>0\}.
\end{eqnarray}

To prove the point~2 of theorem~\ref{gorizont}, it is enough to note that central symmetry $i_0$ of the space-time
$\mathbb{R}^{n+1}$ relative to the origin of coordinate system is an isometry of de Sitter space-time,
reversing the time direction. Therefore $i_0(J^-(p))=J^+(-p)$, $i_0(J^+(p))=J^-(-p)$ for any event $p\in L$.
Corresponding equalities in the point~2 of theorem\,\ref{gorizont} follow from here.

On the ground of p.~2 in theorem\,\ref{gorizont} and (\ref{m7}), (\ref{m8}),
\begin{eqnarray}\label{m9}
J^{+}(-L)=S(R)\cap\{(x_1,\dots,x_{n},t)\in \mathbb{R}^{n+1}:
x_1-t<0\},
\end{eqnarray}
\begin{eqnarray}\label{m10}
J^{-}(-L)=S(R)\cap\{(x_1,\dots,x_{n},t)\in \mathbb{R}^{n+1}:
x_1+t<0\}.
\end{eqnarray}

The statements of p.~3 in theorem\,\ref{gorizont} are valid in view of (\ref{m2}), (\ref{m6}) -- (\ref{m10}).

Let us prove p.~4 of theorem~\ref{gorizont}. It follows from (\ref{m7}), (\ref{m8}) that every set $J^-(L)$, $J^+(L)$
is open and contains no pair of antipodal events. Conversely, in consequence of
(\ref{m2}), (\ref{m6}), the past event horizon $\Gamma^{-}(L)$ and the future event horizon $\Gamma^{+}(L)$ for $L$
are centrally symmetric sets. Therefore the quotient map $pr: S(R)\rightarrow S^1_n(R)$, identifying antipodal events
from $S(R),$  is a diffeomorphism on the open submanifold $J^-(L)$ ($J^+(L)$) and glues antipodal events of the set
$\Gamma^{-}(L)$ ($\Gamma^{+}(L)$). Now the rest statements of p.\,4 in theorem~\ref{gorizont} follow from equalities
in p.~2 of this theorem.

P.~5 of theorem\,\ref{gorizont} is an immediate corollary of statements in p.~4.
\end{proof}

\begin{corollary}\label{sl}
Let $L$ be a time-like geodesic in $S(R)$, $p$ is the joint event of $L$ and $S(R,0)$. Then the past event horizon
$\Gamma^{-}(L)$ for $L$ intersects $S(R,0)$ by the sphere $S_{S(R,0)}(p,\pi R/2)$ of the radius
$\pi R/2$ with the center $p$.
\end{corollary}

\begin{proof}
Using the argument in the proof of theorem\,\ref{gorizont}, we can assume without loss of generality that $L$ is
integral curve of vector field $Y$, intersecting $S(R,0)$ at the point $p$ with Descartes coordinates
$(R,0,\ldots,0)$. Now the corollary~\ref{sl} follows from (\ref{m2}) and the fact that
$$S_{S(R,0)}(p,\pi R/2)=\{(x_1,\ldots,x_n,t)\in S(R)\mid\,x_1=0,\,t=0\}.$$
\end{proof}

One can consider the time-like geodesic $L$ above as the history (or the world line) of an \textit{eternal} observer.

The Fig.~4.18, p. 120 in the Hawking's book~\cite{Hokrus} (Fig.\,1 in our paper) admits two interpretations, namely as
a picture of the history and corresponding (past) event horizon for \textit{a real} or \textit{an eternal} observer
in de Sitter space-time.

The first interpretation corresponds to the inscription "Surface of constant time" but then contradicts to the smoothness
of bright region at its top point since the top point must be the cone point for this region. To avoid the last mistake on
Fig.~4.18, it would better to take the second interpretation and change the inscription "Surface of constant time" by
the inscription ''$t=\infty$'', assuming that the scale on the picture goes to zero when $t\rightarrow \infty.$
But for the second interpretation, the observer's event horizon is depicted incorrectly since on the ground of the
corollary~\ref{sl}, it must intersect the <<throat>>\ of the hyperboloid of one sheet (the sphere $S(R,0)$) by the sphere
$S_{S(R,0)}(p,\pi R/2)$, where $p$ is the intersection event of the world line $L$ with $S(R,0)$; for the first
interpretation, the above intersection must be $S_{S(R,0)}(p,r)$, where $r < \pi R/2$. On the other hand, this
intersection on Fig.~4.18 is empty.

The correct picture for the second interpretation for bounded (respectively, infinite) time is given on our Fig.~2
(respectively, Fig.~3). Note also that the observer's event horizon (see Fig.\,2) consists of all isotropic
geodesics, lying in corresponding hyperspace in \textit{Mink}$^{\,n+1}$ going through zero event.

Earlier V.N.Berestovskii formulated without proof the above statements about past event horizon
in his plenary talk, but in the text \cite{Ber1} of this talk they are absent.

\newpage

\begin{figure}[h]
\includegraphics[width=13.5cm]{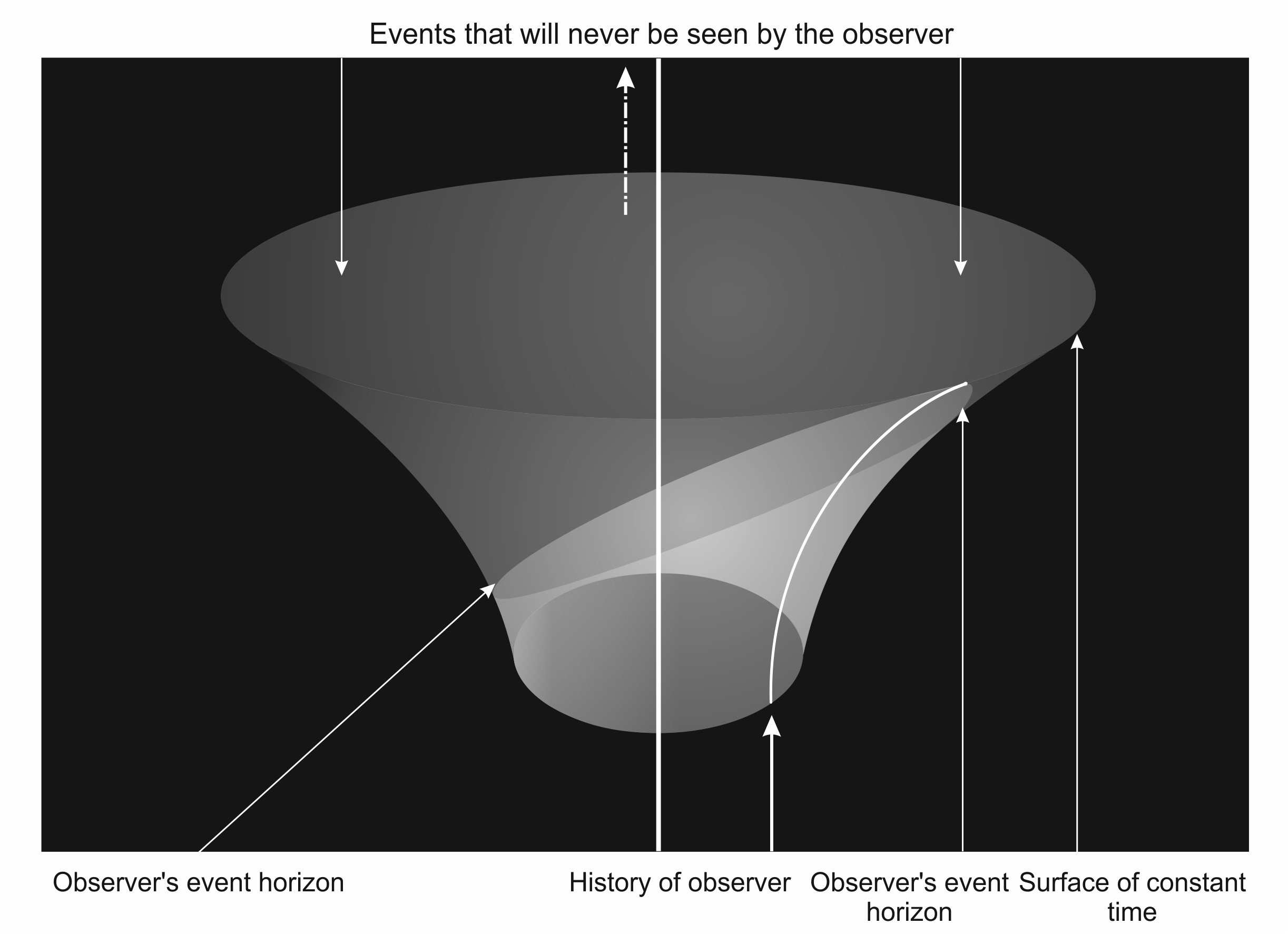}
\caption{Fig.~4.18 of (real) observer's event horizon from Hawking's book~\cite{Hokrus}}
\label{fig:fig1}
\end{figure}

\begin{figure}[h]
\includegraphics[width=13.5cm]{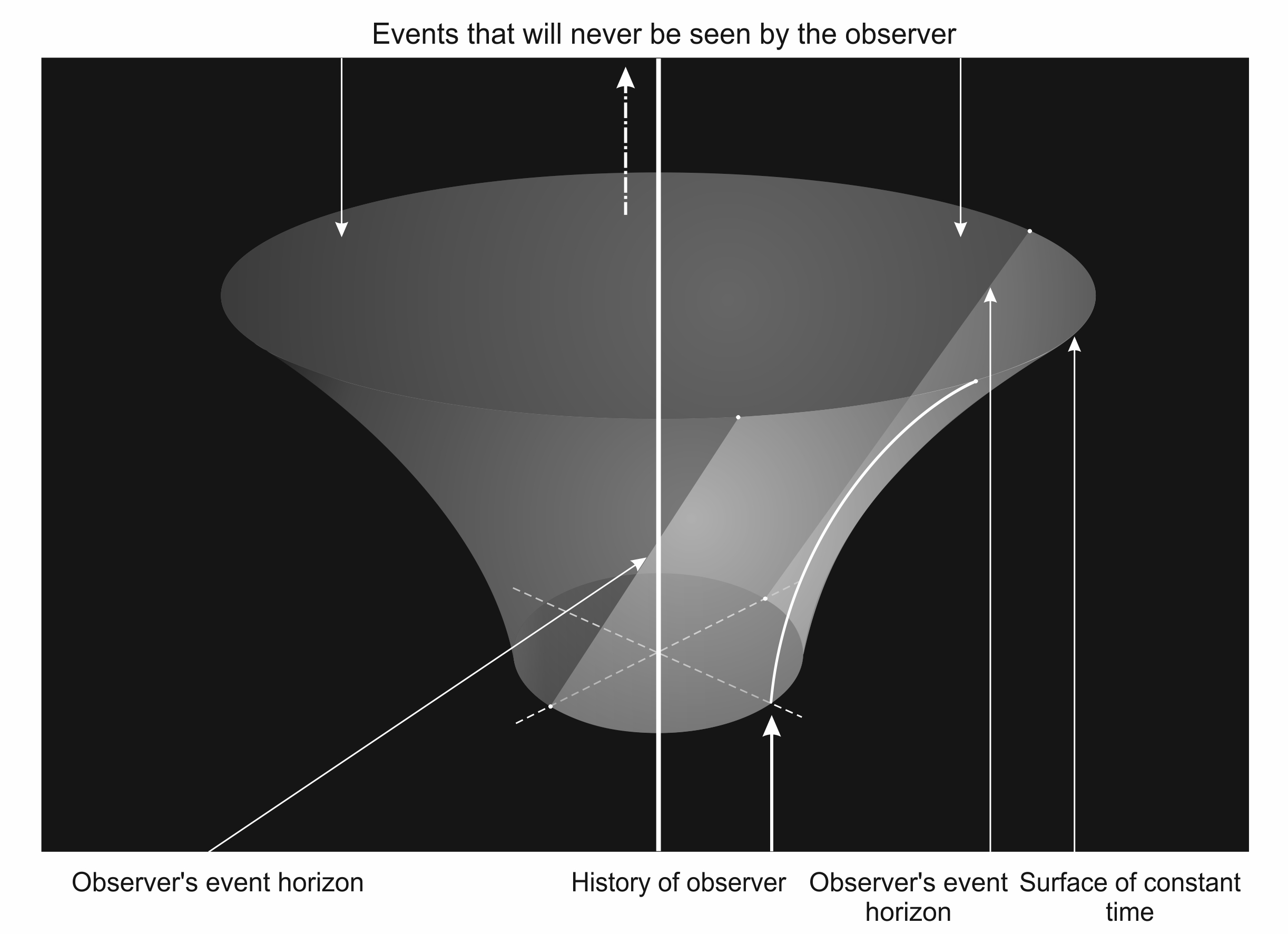}
\caption{Correct (eternal) observer's event horizon for $0\leq t \leq t_0$}
\label{fig:fig2}
\end{figure}

\begin{figure}[h]
\includegraphics[width=13.5cm]{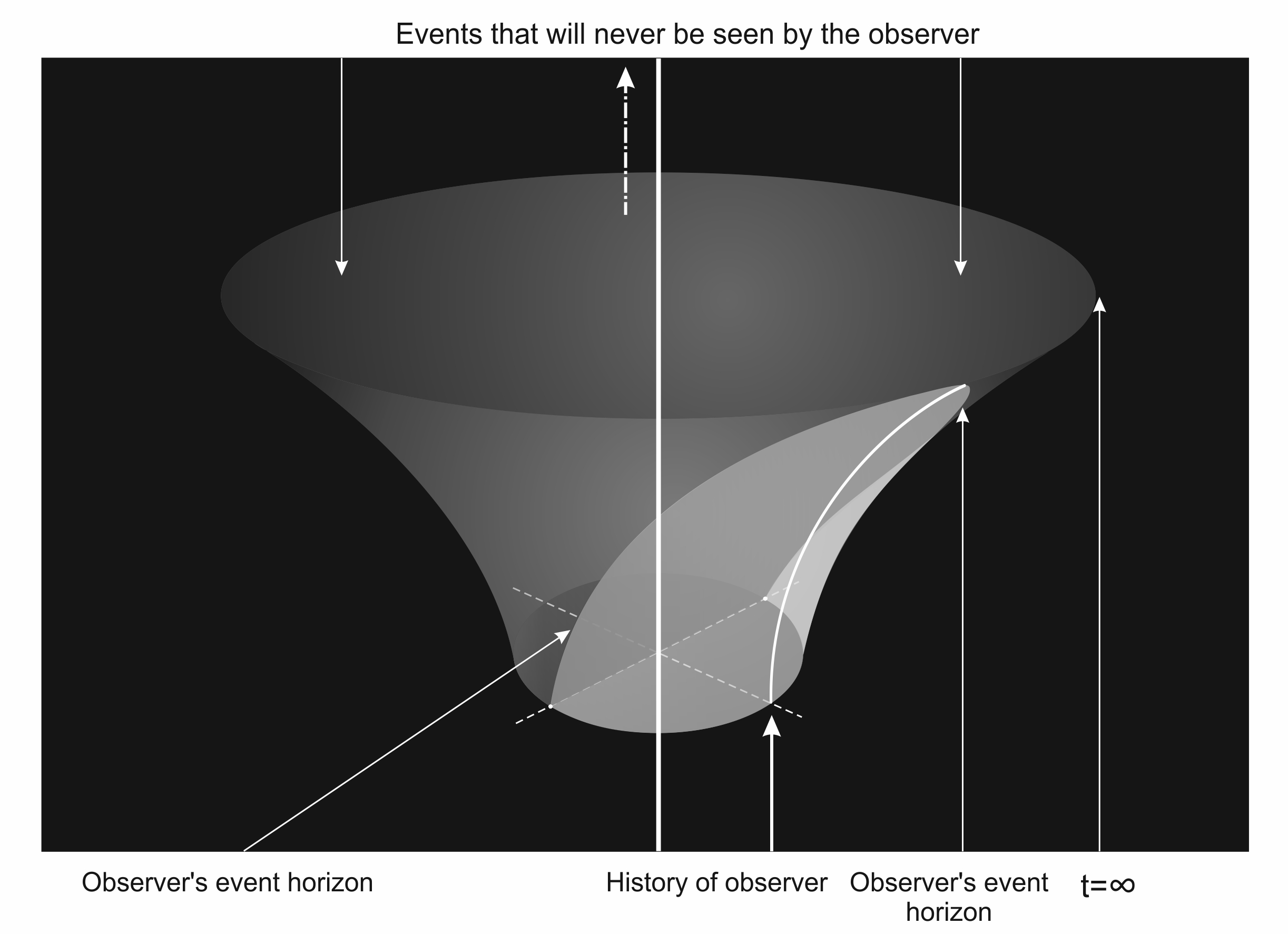}
\caption{(Eternal) observer's event horizon for $0\leq t \leq \infty$}
\label{fig:fig3}
\end{figure}

\vspace{5mm}

Berestovskii Valerii Nikolaevich, \newline Sobolev Institute Of mathematics SB RAS,
Omsk department
\newline 644099, Omsk, Pevtsova street, 13, Russia

%\emph{E-mail address}: \texttt{berestov@ofim.oscsbras.ru}

\vspace{2mm}

Zubareva Irina Aleksandrovna, \newline Sobolev Institute Of mathematics SB RAS,
Omsk department
\newline 644099, Omsk, Pevtsova street, 13, Russia

%\emph{E-mail address}: \texttt{i\_gribanova@mail.ru}

\end{document}